\documentclass[
12pt,twoside, a4paper,
]{amsart}

\title[Weak approximation]
{           {\protect\hfill \normalfont \tiny
            \\ \vspace{10pt}}
The defect of weak approximation\\
for homogeneous spaces. II\\
}
\author{Mikhail Borovoi}

\address{
Raymond and Beverly Sackler School of Mathematical Sciences,
Tel Aviv University, 69978 Tel Aviv, Israel}
\email{borovoi@post.tau.ac.il}

\thanks{Partially supported by the
Hermann Minkowski Center for Geometry
and by ISF grant 807/07}


\usepackage[arrow,matrix]{xy}
\usepackage{mathptmx}
\usepackage{amscd}
\usepackage{amssymb}
\usepackage{eucal}
\usepackage{amsxtra}
\usepackage{verbatim}

\theoremstyle{plain}

\newtheorem{theorem}{Theorem}[section]
\newtheorem{proposition}[theorem]{Proposition}
\newtheorem{lemma}[theorem]{Lemma}
\newtheorem{corollary}[theorem]{Corollary}

\newtheorem{conditional-result}[theorem]{Conditional Result}

\newtheorem{theorem?}{Theorem(?)}[section]
\newtheorem{proposition?}[theorem]{Proposition(?)}
\newtheorem{lemma?}[theorem]{Lemma(?)}
\newtheorem{corollary?}[theorem]{Corollary(?)}

\newtheorem*{theorem*}{Theorem}
\newtheorem*{proposition*}{Proposition}
\newtheorem*{lemma*}{Lemma}
\newtheorem*{corollary*}{Corollary}
\newtheorem*{question*}{Question}
\newtheorem*{conjecture*}{Conjecture}
\newtheorem*{claim*}{Claim}

\newtheorem*{introtheorem*}{Theorem}
\newtheorem*{introproposition*}{Proposition}
\newtheorem*{introlemma*}{Lemma}
\newtheorem*{introcorollary*}{Corollary}

\theoremstyle{definition}

\newtheorem{definition}[theorem]{Definition}

\newtheorem*{definition*}{Definition}
\newtheorem*{example*}{Example}

\newtheorem{subsec}[theorem]{}

\theoremstyle{remark}

\newtheorem*{remark*}{Remark}

\newtheorem*{acknowledgements}{Acknowledgements}

\DeclareSymbolFont{rsfs}{U}{rsfs}{m}{n}
\DeclareSymbolFontAlphabet{\mathcal}{rsfs}

\DeclareFontEncoding{OT2}{}{} 
\DeclareTextFontCommand{\textcyr}{\fontencoding{OT2}
    \fontfamily{wncyr}\fontseries{m}\fontshape{n}\selectfont}
\newcommand{\Sh}{\textcyr{Sh}}
\newcommand{\Ch}{\textcyr{Ch}}

\newcommand{\sV}{{\mathcal{V}}}
\newcommand{\Vinf}{{\sV_\infty}}

\newcommand{\isoto}{\overset{\sim}{\to}}

\newcommand{\labelto}[1]{\xrightarrow{\makebox[1.5em]{\scriptsize ${#1}$}}}

\def\uu{^\mathrm{u}}
\def\red{^\mathrm{red}}
\def\tor{^{\mathrm{tor}}}

\def\sss{^{\mathrm{ss}}}

\def\ssu{^{\mathrm{ssu}}}

\newcommand{\ZZ}{{\mathbb{Z}}}
\newcommand{\QQ}{{\mathbb{Q}}}

\newcommand{\CT}{J.-L.~Colliot-Th\'el\`ene}

\newcommand{\Gal}{{\rm Gal}}
\newcommand{\Pic}{{\rm Pic}}

\renewcommand{\ker}{{\rm ker}}
\newcommand{\coker}{{\rm coker}}

\newcommand{\Hom}{{\rm Hom}}

\newcommand{\kbar}{{\overline{k}}}

\newcommand{\WAS}{({\rm WA}_S)}

\newcommand{\XX}{{\mathbf{X}}}

\newcommand{\bs}{{\backslash}}
\newcommand{\pn}{\par\noindent}

\newcommand{\CSHG}{{C_S(H,G)}}
\newcommand{\CSH}{{C_S(H)}}

\newcommand{\loc}{{\rm loc}}
\newcommand{\im}{{\rm im}}

\newcommand{\Bbar}{{\overline{B}}}
\newcommand{\Gbar}{{\overline G}}

\newcommand{\That}{{\hat{T}}}
\newcommand{\Hhat}{{\hat{H}}}

\renewcommand{\gg}{{\mathfrak{g}}}
\newcommand{\hh}{{\mathfrak{h}}}
\newcommand{\Aut}{{\rm Aut}}
\newcommand{\Res}{{\rm Res}}

\begin{document}

\begin{abstract}
Let $X$ be a right homogeneous space of  a connected linear algebraic group  $G'$
over a number field $k$, containing a $k$-point $x$.
Assume that the stabilizer of $x$ in $G'$ is connected.
Using the notion of a quasi-trivial group, recently introduced
by Colliot-Th\'el\`ene, we can represent $X$ in the form $X=H\backslash G$,
where $G$ is a quasi-trivial $k$-group and $H\subset G$ is a connected
$k$-subgroup.

Let $S$ be a finite set of places of $k$. Applying results of \cite{B2},
we compute the defect of weak approximation for $X$ with respect to $S$
in terms of the biggest toric quotient $H^{\rm tor}$ of $H$.
In particular, we show that if $H^{\rm tor}$ splits over a metacyclic extension of $k$,
then $X$ has the weak approximation property.
We show also that any homogeneous space $X$ with connected stabilizer
(without assumptions on $H^{\rm tor}$) has the real approximation property.
\end{abstract}

\maketitle

\section{Introduction}

This note is a sequel for \cite{B2},
and we use the notation of that paper.
Let $k$ be a number field, and let $\kbar$ be a fixed algebraic closure of $k$.
We write $\sV$ for the set of all places of $k$,
and $\Vinf$ for the set of its archimedean places.
If $v\in\sV$, we write $k_v$ for the completion of $k$ at $v$.

Let $X$ be an algebraic variety over $k$.
We refer to \cite{B2} for preliminaries on weak approximation for $X$.
If $S\subset\sV$ is a finite set of places,
we write $\WAS$ for the weak approximation property
with respect to $S$.
Thus, ``$X$ has $\WAS$'' means that $X(k)$ is dense in $\prod_{v\in S} X(k_v)$.
We say that $X$ has \emph{the weak approximation property,} if $X$ has $\WAS$
for any finite subset $S\subset\sV$.
We say that $X$ has \emph{the real approximation property,}
if $X$ has the weak approximation property $\WAS$ with respect to $S=\Vinf$.

In \cite{B2} we considered the case $X=H\bs G$, where $H\subset G$
is a connected $k$-subgroup of a connected $k$-group $G$,
assuming that $\Sh(G)=0$ and $A(G)=0$
(the assumption $A(G)=0$ means that $G$ has the weak approximation property).
Under these assumptions we constructed a certain abelian group $C_S(H,G)$
which is the defect of weak approximation for $X$ with respect to $S$:
the variety $X$ has $\WAS$ if and only if $C_S(H,G)=0$.
We initially constructed $\CSHG$ in terms of $H$ and $G$,
but then we computed it in terms of the Brauer group of $X$.

In the present note we consider the case of an arbitrary homogeneous space
with connected stabilizer $X=H'\bs G'$,
where $G'$ is any connected linear $k$-group and $H'\subset G'$ is a connected $k$-subgroup.
Using the notions of a quasi-trivial $k$-group and a flasque resolution,
introduced by \CT\ \cite{CT}, we notice that we can represent $X$ in the form
$X=H\bs G$, where $G$ is a \emph{quasi-trivial} group and $H\subset G$ is a connected $k$-subgroup
(Lemma \ref{lem:repr-quasi-trivial}).
We have $\Sh(G)=0$ and $A(G)=0$, because $G$ is quasi-trivial.
Now we can apply \cite[Theorem 1.3]{B2}.
We obtain that $X$ has $\WAS$ if and only if $\CSHG=0$.

Moreover, we have $\Pic(G_K)=0$ for any field extension $K/k$, because $G$ is quasi-trivial.
Using this fact, we show that the group $\CSHG$ can be computed in terms of $H$ only.
Namely, we construct a group $\CSH$ in terms of $H$ as in \cite{B1}
and prove that $\CSHG=\CSH$ (Lemma \ref{lem:CSHG=CSH}).

We see that $X$ has $\WAS$ if and only if $\CSH=0$.
We say that $\CSH$ is \emph{the defect of weak approximation for $X$ with respect to $S$.}
Note that the group $\CSH$ does not depend on the representation of $X$ in the form $X=H\bs G$
with quasi-trivial $G$ and connected $H$,
because it can be computed in terms of the Brauer group of $X$ (\cite[Theorem 1.11]{B2}).

Let $H\tor$ denote the biggest quotient torus of $H$.
We show that the canonical homomorphism $\CSH\to C_S(H\tor)$ is an isomorphism
(Proposition \ref{prop:Htor-B1}).
It follows that $X$ has $\WAS$ if and only if $C_S(H\tor)=0$.
We notice that
$$
C_S(H\tor)\simeq\coker\left[H^1(k,H\tor)\to\prod_{v\in S} H^1(k_v,H\tor)\right].
$$

Let $L/k$ be a Galois extension splitting the torus $H\tor$.
Let $S_0$ denote the set of (nonarchimedean, ramified in $L$) places $v$ of $k$
such that the decomposition group of $v$ in $\Gal(L/k)$ is noncyclic.
We prove that $C_S(H)=C_{S\cap S_0}(H)$  (Corollary \ref{cor:S0}).

Assume that $S\cap S_0=\emptyset$, i.e. all the places in $S$ have
cyclic decomposition subgroups in $\Gal(L/k)$.
Then $\CSH=0$, hence $X$ has $\WAS$ (Theorem \ref{thm:ScapS0=emptyset}).
In particular, $C_{\sV_\infty}(H)=0$ for any $H$.
Thus any homogeneous space $X$ of a connected $k$-group
with connected stabilizer has the real approximation property
(Corollary \ref{cor:real-approximation}).

Now assume that $H\tor$ splits over a cyclic extension of $k$
(e.g. $H\tor=1$).
Then $S_0=\emptyset$, hence $\CSH=0$ for any $S$,
and $X$ has the weak approximation property (Corollary \ref{cor:Htor-splits-cyclic}).
Moreover, we prove that if $H\tor$ splits over a metacyclic extension,
then $X$ has the weak approximation property (Theorem \ref{thm:Htor-metacyclic}).

These results generalize the results of \cite{B1},
where we assumed that $G$ is semisimple simply connected.
They also generalize results of Sansuc \cite{Sansuc}
on weak approximation for connected linear groups.

We could state and prove our results thanks to the notion
of a quasi-trivial group introduced by Colliot-Th\'el\`ene \cite{CT}.
The constructions and proofs are based on results of Kottwitz \cite{K}.
Of course, our results are based on the classical results of
Kneser, Harder, Chernousov, and Platonov on the Hasse principle
and weak approximation for simply connected semisimple groups.

\begin{acknowledgements}
This note was written when the author was visiting the Max-Planck-Institut f\"ur Mathematik,
Bonn (MPIM). The author is grateful to MPIM for hospitality, support,
and excellent working conditions.
The author is grateful to Boris Kunyavski\u\i\ for useful discussions
and for help in proving Theorem \ref{thm:Htor-metacyclic}.
\end{acknowledgements}


\section{Preliminaries on quasi-trivial groups}
The results of this section are actually due to \CT\ \cite{CT}.

\begin{subsec}
Let $k$ be a field of characteristic 0, $\kbar$ a fixed algebraic closure of $k$.
Let $G$ be a connected linear $k$-group.
We set $\Gbar=G\times_k \kbar$.
We use the following notation:
\pn $G\uu$ is the unipotent radical of $G$;
\pn $G\red=G/G\uu$ (it is reductive);
\pn $G\sss$ is the derived group of $G\red$ (it is semisimple);
\pn $G\tor=G\red/G\sss$ (it is a torus);
\pn $G\ssu=\ker[G\to G\tor]$ (it is an extension of $G\sss$ by $G\uu$).
\end{subsec}

\begin{definition}[\CT]
A connected linear $k$-group $G$ over a field $k$ of characteristic 0
is called \emph{quasi-trivial},
if $G\tor$ is a quasi-trivial torus and $G\sss$ is simply connected.
\end{definition}

Recall that a $k$-torus $T$ is called quasi-trivial
if its character group $\XX(\overline{T})$ is a permutation
$\Gal(\kbar/k)$-module.

Note that if $G$ is quasi-trivial, then for any field extension $K/k$
the group $G_K$ is quasi-trivial.

\begin{lemma}\label{lem:Pic}
Let $G$ be a quasi-trivial group over a field $k$ of characteristic 0.
Then $\Pic(G)=0$, where $\Pic$ denotes the Picard group.
\end{lemma}

\begin{proof}
If
$$
1\to G'\to G\to G''\to 1
$$
is a short exact sequence of connected linear $k$-groups,
then we have an exact sequence
\begin{equation}\label{eq:exact-groups}
\XX(G')\to\Pic(G'')\to\Pic(G)\to\Pic(G'),
\end{equation}
where $\XX(G')$ denotes the  group of $k$-characters of $G'$,
see \cite[Corollary 6.11]{Sansuc}.

Since $G\uu$ is a unipotent $k$-group,  the exponential map
${\rm exp}\colon {\rm Lie\,}G\uu\to G\uu$ is a biregular isomorphism
of algebraic varieties (because ${\rm char}(k)=0$),
hence $\Pic(G\uu)=0$.
By \cite[Lemme 6.9]{Sansuc} $\Pic(G\sss)=0$ (because $G\sss$ is simply connected)
and $\Pic(G\tor)=H^1(k,\XX(\Gbar\tor))$.
Since $\XX(\Gbar\tor)$ is a permutation module, we see that $\Pic(G\tor)=0$.
Using exact sequence \eqref{eq:exact-groups}, we conclude by d\'evissage
that $\Pic(G)=0$.
\end{proof}

\begin{lemma}\label{lem:Sh-A}
Let $G$ be a quasi-trivial $k$-group over a number field $k$.
Then $\Sh(G)=0$ and $A(G)=0$.
\end{lemma}

\begin{proof}
By \cite[Proposition 9.2]{CT}
we have $\Sh(G\red)=0$ and $A(G\red)=0$.
By \cite[Proposition 4.1]{Sansuc} $\Sh(G)=\Sh(G\red)$.
By \cite[Proposition 3.2]{Sansuc} $A(G)=A(G\red)$.
Thus $\Sh(G)=0$ and $A(G)=0$.
\end{proof}

\begin{lemma}\label{lem:repr-quasi-trivial}
Let $k$ be a field of characteristic 0 and $X$
a right homogeneous space with connected stabilizer over $k$,
i.e. $X=H'\bs G'$, where $G'$ is a connected linear $k$-group
and $H'\subset G'$ is a connected $k$-subgroup.
Then one can represent $X$ as $X=H\bs G$,
where $G$ is a quasi-trivial $k$-group and $H\subset G$ is a connected $k$-subgroup.
\end{lemma}

\begin{proof}
By \cite[Proposition-D\'efinition 3.1]{CT}
there exists a flasque resolution of $G'$, i.e. a central extension of connected $k$-groups
$$
1\to F\to G\to G'\to 1,
$$
where $G$ is quasi-trivial and $F$ is a flasque $k$-torus.
Let $H$ be the preimage of $H'$ in $G$.
From the exact sequence
$$
1\to F\to H\to H'\to 1
$$
we see that $H$ is connected, because $H'$ and $F$ are connected.
We have $X=H\bs G$.
\end{proof}


\section{Defect of weak approximation}\label{sec:defect}

\begin{subsec}\label{subsec:repr-quasi-trivial}
Let $X$ be a homogeneous space with connected stabilizer over a number field $k$,
i.e. $X=H'\bs G'$, where $G'$ is a connected linear $k$-group and $H'\subset G'$
is a connected $k$-subgroup.
By Lemma \ref{lem:repr-quasi-trivial} we may write $X=H\bs G$, where $G$ is a quasi-trivial $k$-group
and $H\subset G$ is a connected $k$-subgroup.

By Lemma \ref{lem:Sh-A} $\Sh(G)=0$ and $A(G)=0$.
Therefore we can apply the results of \cite{B2}.
\end{subsec}

\begin{subsec}
Let $X,\ G,\ H$ be as in \ref{subsec:repr-quasi-trivial}.
Let $S\subset\sV$ be a finite subset.
Set
\begin{align*}
&B(H)=\Hom(\Pic(H),\QQ/\ZZ)=(\pi_1(H)_\Gamma)_{{\rm tors}}\\
&B_v(H)=B(H_{k_v}) \text{ for }v\in\sV
\end{align*}
with the notation of \cite{B2}.
Consider the canonical homomorphism
$$
\lambda_v\colon B_v(H)\to B(H).
$$
Set:
\begin{align*}
&B^S(H)=\langle\lambda_v(B_v(H))\rangle_{v\in \sV\smallsetminus S}\\
&B'(H)= B^\emptyset(H)= \langle\lambda_v(B_v(H))\rangle_{v\in \sV} \\
&C_S(H)= B'(H)/B^S(G),
\end{align*}
where $\langle\lambda_v(B_v(H))\rangle_{v\in \sV\smallsetminus S}$ denotes the subgroup of $B(H)$
generated by the subgroups $\lambda_v(B_v(H))$ for all $v\in\sV\smallsetminus S$.
\end{subsec}

\begin{subsec}
For a homogeneous space $X=H\bs G$ over $k$, without assuming that $G$ is quasi-trivial,
we defined in \cite{B2} the following groups:
\begin{align*}
&B(H,G)=\ker[B(H)\to B(G)],\\
&B_v(H,G)=B(H_{k_v},G_{k_v})=\ker[B_v(H)\to B_v(G)],
\end{align*}
and also $B^S(H,G)$, $B'(H,G)$, and $C_S(H,G)$, see \cite[Section 1.2]{B2}.
\end{subsec}

\begin{lemma}\label{lem:CSHG=CSH}
Let $k,\ X,\ G,\ H$ be as in \ref{subsec:repr-quasi-trivial}
(in particular $G$ is quasi-trivial).
Then there is a canonical isomorphism $C_S(H,G)\isoto C_S(H)$.
\end{lemma}

\begin{proof}
Since $G$ is quasi-trivial, by Lemma \ref{lem:Pic}
$\Pic(G)=0$, hence $B(G)=0$.
Since $G_{k_v}$ is also quasi-trivial, we see that $B_v(G)=0$.
We obtain successively that
$B(H,G)=B(H),\ B_v(H,G)=B_v(H),\ B^S(H,G)=B^S(H),\ B'(H,G)=B'(H)$,
whence $C_S(H,G)=C_S(H)$.
\end{proof}

\begin{theorem}\label{thm:WAS-CSH=0}
Let $k,\ X,\ G,\ H$ be as in \ref{subsec:repr-quasi-trivial}
(in particular $G$ is quasi-trivial).
Let $S\subset\sV$ be a finite set of places of $k$.
Then $X$ has $\WAS$ if and only if $C_S(H)=0$.
\end{theorem}

\begin{proof}
By Lemma \ref{lem:Sh-A} $\Sh(G)=0$ and $A(G)=0$.
By \cite[Theorem 1.3]{B2} $X$ has $\WAS$ if and only if $C_S(H,G)=0$.
By Lemma  \ref{lem:CSHG=CSH} $\CSHG=\CSH$, and the theorem follows.
\end{proof}

\begin{lemma}\label{lem:Htor}
Let $H$ be a connected linear $k$-group over a number field $k$.
Assume that $H\tor=1$.
Then for any place $v$ of $k$
the map $\lambda_v\colon B_v(H)\to B(H)$ is surjective.
\end{lemma}

\begin{proof}
See \cite[Proof of Theorem 3.4(b)]{CTXu}.
\end{proof}

\begin{proposition}[{\cite[Theorem 1.4]{B1}}]\label{prop:Htor-B1}
Let $H$ be a connected $k$-group over a number field $k$.
Let $S\subset\sV$ be a finite set of places of $k$.
Then the canonical homomorphism $C_S(H)\to C_S(H\tor)$ is an isomorphism.
\end{proposition}

\begin{proof}
Since \cite{B1} is not easily accessible, we reproduce the proof here.

First, consider $H\ssu$.
Since $(H\ssu)\tor=1$,
by Lemma \ref{lem:Htor} for any place $v$ of $k$ we have $\lambda_v(B_v(H\ssu))=B(H\ssu)$.
We see that $B^S(H\ssu)=B'(H\ssu)=B(H\ssu)$.

Consider the canonical short exact sequence
$$
1\to H\ssu\to H\to H\tor\to 1.
$$
Exact sequence \eqref{eq:exact-groups} from the proof of Lemma \ref{lem:Pic}
gives us an exact sequence
$$
\XX(H\ssu)\to\Pic(H\tor)\to\Pic(H)\to\Pic(H\ssu),
$$
where clearly $\XX(H\ssu)=0$.
We obtain the dual exact sequence
$$
B(H\ssu)\to B(H)\to B(H\tor)\to 0
$$
and similar exact sequences for the groups $B_v$.
Since $B^S(H\ssu)=B(H\ssu)$,
we obtain an exact sequence
$$
B(H\ssu)\to B^S(H)\to B^S(H\tor)\to 0.
$$

Set $\Bbar=\im[B(H\ssu)\to B(H)]$, then we obtain an exact sequence
$$
0\to\Bbar\to B^S(H)\to B^S(H\tor)\to 0
$$
and a commutative diagram with exact rows
\[
\xymatrix{
0\ar[r]  &{\Bbar} \ar@{=}[d] \ar[r]  &B^S(H) \ar@{^{(}->}[d] \ar[r]  &B^S(H\tor)\ar@{^{(}->}[d] \ar[r]  &0\\
0\ar[r]  &{\Bbar}\ar[r]              &B'(H) \ar[r]                   &B'(H\tor) \ar[r]                  &0
}
\]
Now the snake lemma gives us an isomorphism $C_S(H)=B'(H)/B^S(H)\isoto B'(H\tor)/B^S(H\tor)=C_S(H\tor)$.
\end{proof}

\begin{corollary}\label{cor:Htor=1}
Let $k,\ X,\ G,\ H$ be as in \ref{subsec:repr-quasi-trivial}
(in particular $G$ is quasi-trivial and $H$ is connected).
Assume that $H\tor=1$.
Then $X$ has the weak approximation property.
\end{corollary}

\begin{proof}
By Proposition \ref{prop:Htor-B1} we have $C_S(H)=C_S(H\tor)=0$ for any $S$.
By Theorem \ref{thm:WAS-CSH=0} $X$ has $\WAS$ for any $S$.
\end{proof}

\begin{remark*}
In the case when $G$ is semisimple simply connected,
this result was proved in \cite[Corollary 1.7]{B1}.
For a simple proof see \cite[Theorem 3.4(b)]{CTXu}.
\end{remark*}

The following result relates $C_S(H)$
to the Galois cohomology of $H\tor$.

\begin{proposition}\label{prop:torus}
Let $T$ be a $k$-torus over a number field $k$.
Let $S\subset\sV$ be a finite set of places of $k$.
Then there is a canonical isomorphism
$$
C_S(T)\isoto \coker\left[ H^1(k,T)\to\prod_{v\in S} H^1(k_v,T)\right].
$$
\end{proposition}

\begin{proof}
We have canonical  duality isomorphisms
$$
\beta_v\colon H^1(k_v,T)\isoto \Hom(H^1(k_v,\XX(\overline{T})),\QQ/\ZZ) =  B_v(T),
$$
cf. \cite[Chapter I, Corollary 2.3 and Theorem 2.13]{Milne}.
Moreover, we have an exact sequence
\begin{equation}\label{eq:exact-Milne}
H^1(k,T)\labelto{\loc} \oplus_{v\in\sV} H^1(k_v,T)\labelto{\mu}B(T),
\end{equation}
where $\loc$ is the localization map, $\mu((\xi_v)_{v\in\sV})=\sum\mu_v(\xi_v)$,
and $\mu_v$ is the composed map
$$
\mu_v\colon H^1(k_v,T)\labelto{\beta_v} B_v(T)\labelto{\lambda_v} B(T),
$$
cf. \cite[Chapter I, Theorem 4.20(b)]{Milne}.

Consider the localization map $\loc_S\colon H^1(k,T)\to\prod_{v\in S}H^1(k_v,T)$.
Let $\xi_S=(\xi_v)_{v\in S}\in \prod_{v\in S}H^1(k_v,T)=\oplus_{v\in S}B_v(T)$,
where we identify $H^1(k_v,T)$ with $B_v(T)$ using $\beta_v$.
From exact sequence \eqref{eq:exact-Milne} we see that $\xi_S\in\im\ \loc_S$
if and only if there exists an element $\xi^S\in \oplus_{v\notin S}B_v(T)$
such that $\mu(\xi_S,\xi^S)=0$.
Such an element $\xi^S$ exists if and only if
$$
\sum_{v\in S}\mu_v(\xi_v)\in B^S(T)\subset B(T).
$$

Set $B_S(T)=\langle\lambda_v(B_v(T))\rangle_{v\in S}$.
Then we see that there is a canonical isomorphism
\begin{align*}
\coker\left[ H^1(k,T)\to \prod_{v\in S} H^1(k_v,T)\right]\isoto B_S(T)/\left(B_S(T)\cap B^S(T)\right)\simeq&\\
    \simeq\left(B_S(T)+B^S(T)\right)/B^S(T)=B'(T)/B^S(T)=&C_S(T).
\end{align*}
\end{proof}

\begin{proposition}\label{prop:torus-S0}
Let $T$ be a $k$-torus over a number field $k$.
Let $L/k$ be a Galois extension splitting $T$.
Let $S_0$ be the set of (nonarchimedean, ramified in $L$)
places $v$ of $k$ whose decomposition groups in $\Gal(L/k)$ are noncyclic.
Let $S\subset \sV$ be any finite set of places of $k$.
Then the canonical homomorphism $C_S(T)\to C_{S\cap S_0}(T)$ is an isomorphism.
\end{proposition}

\begin{proof}
Let $v\in S$.
Let $w$ be a place of $L$ lying over $v$.
Let $D_w\subset\Gal(L/k)$ be the decomposition group of $w$.
Then by \cite[Lemme 6.9]{Sansuc} $\Pic(T_{k_v})=H^1(D_w,\XX(T_L)$.
We see that the image $\lambda_v(B_v(T))\subset B(T)$
depends only on the conjugacy class of $D_w\subset\Gal(L/k)$.

If $v\in S$, $v\notin S_0$, then $D_w$ is cyclic for $w$ lying over $v$.
By Chebotarev's density theorem there exists $v'\notin S$
and $w'$ lying over $v'$ such that $D_{w'}=D_w$.
It follows that $\lambda_v(B_v(T))=\lambda_{v'}(B_{v'}(T))$.
But $v'\notin S$, hence $\lambda_{v'}(B_{v'}(T))\subset B^S(T)$.
We see that $\lambda_{v}(B_{v}(T))\subset B^S(T)$.
Thus $B^{S\cap S_0}(T)=B^S(T)$.
We conclude that $C_{S\cap S_0}(T)=C_{S}(T)$.
\end{proof}

\begin{corollary}\label{cor:S0}
Let $H$ be a connected linear $k$-group
over a number field $k$.
Let $L/k$ be a Galois extension splitting $H\tor$.
Let $S_0$ be the set of 
places $v$ of $k$ whose decomposition groups in $\Gal(L/k)$ are noncyclic.
Let $S\subset \sV$ be any finite set of places of $k$.
Then the canonical homomorphism $C_S(H)\to C_{S\cap S_0}(H)$ is an isomorphism.
\end{corollary}

\begin{proof}
We have a commutative diagram of canonical homomorphisms
$$
\xymatrix{
C_S(H)\ar[d]\ar[r]^\simeq          &C_S(H\tor)\ar[d]^\simeq\\
C_{S\cap S_0}(H)\ar[r]^-\simeq      &C_{S\cap S_0}(H\tor)
}
$$
By Proposition \ref{prop:Htor-B1} the horizontal arrows are isomorphisms.
By Proposition \ref{prop:torus-S0} the right vertical arrow is an isomorphism.
We conclude that the left vertical arrow is also an isomorphism.
\end{proof}

\begin{theorem}\label{thm:ScapS0=emptyset}
Let $k,\ X,\ G,\ H$ be as in \ref{subsec:repr-quasi-trivial}
(in particular $G$ is quasi-trivial and $H$ is connected).
Let $L/k$ be a Galois extension splitting $H\tor$.
Let $S_0$ be the set of 
places $v$ of\ $k$ whose decomposition groups in $\Gal(L/k)$ are noncyclic.
Let $S\subset\sV$ be a finite set of places of $k$ such that $S\cap S_0=\emptyset$.
Then $X$ has $\WAS$.
\end{theorem}

\begin{proof}
By Corollary \ref{cor:S0}  $C_S(H)=C_{S\cap S_0}(H)=C_\emptyset(H)=0$.
By Theorem \ref{thm:WAS-CSH=0} $X$ has $\WAS$.
\end{proof}

\begin{corollary}\label{cor:real-approximation}
Let $X$ be as in \ref{subsec:repr-quasi-trivial}.
Then $X$ has the real approximation property.
\end{corollary}

\begin{proof}
Let $L$ and $S_0$ be as in Theorem \ref{thm:ScapS0=emptyset}.
Take $S=\Vinf$.
A decomposition group of an archimedean place is either 0 or $\ZZ/2$,
hence cyclic.
We see that $\Vinf\cap S_0=\emptyset$.
By Theorem \ref{thm:ScapS0=emptyset} $X$ has $({\rm WA}_\Vinf)$,
i.e. $X$ has the real approximation property.
\end{proof}

\begin{proof}[Another proof]
The subgroup $H$ is connected, and we have $\Sh(G)=0$ and $A(G)=0$.
Now by \cite[Corollary 1.7]{B2} $X$ has the real approximation property.
\end{proof}

\begin{corollary}\label{cor:Htor-splits-cyclic}
Let $k,\ X,\ G,\ H$ be as in \ref{subsec:repr-quasi-trivial}
(in particular $G$ is quasi-trivial and $H$ is connected).
Assume that $H\tor$ splits over a \emph{cyclic} extension $L$ of $k$.
Then $X$ has the weak approximation property.
\end{corollary}

\begin{proof}
Let $S_0$ denote the set of
places $v$ of $k$ whose decomposition groups in $\Gal(L/k)$ are noncyclic,
then $S_0=\emptyset$.
Thus for any finite $S\subset \sV$ we have $S\cap S_0=\emptyset$.
By Theorem \ref{thm:ScapS0=emptyset}  $X$ has $\WAS$ for any $S$,
i.e. $X$ has the weak approximation property.
\end{proof}


\section{Metacyclic extensions}

In this section, inspired by \cite[Lemme 1.3]{Sansuc},
we generalize Corollary \ref{cor:Htor-splits-cyclic}.

\begin{subsec}
Recall that a finite group is called metacyclic if
all its Sylow subgroups are cyclic.
For example, the symmetric group $S_3$ is metacyclic,
while the group $\ZZ/2\oplus\ZZ/2$ is not.
Every cyclic group is metacyclic.
We say that a Galois extension $L/k$ is metacyclic
if $\Gal(L/k)$ is a metacyclic group.
\end{subsec}

\begin{theorem}\label{thm:Htor-metacyclic}
Let $k,\ X,\ G,\ H$ be as in \ref{subsec:repr-quasi-trivial}
(in particular $G$ is quasi-trivial and $H$ is connected).
Assume that $H\tor$ splits over a \emph{metacyclic} extension $L$ of $k$.
Then $X$ has the weak approximation property.
\end{theorem}

\begin{proof}
Set $T=H\tor$, then $T$ is a $k$-torus splitting over $L$.
By Theorem \ref{thm:WAS-CSH=0} and Proposition \ref{prop:Htor-B1}
it suffices to prove that $C_S(T)=0$.
Set
$$
\Ch^1_S(T)=\coker\left[H^1(k,T)\to\prod_{v\in S} H^1(k_v,T)\right].
$$
By Proposition \ref{prop:torus} $C_S(T)\simeq\Ch^1_S(T)$.

We write $\That$ for $\XX(\overline{T})$.
Set
\begin{align*}
\Sh^1_S(k,\That)&=\ker\left[ H^1(k,\That)\to\prod_{v\in\sV\smallsetminus S}H^1(k_v,\That)\right]\\
\Sh^1_\omega(k,\That)&=\bigcup_S \Sh^1_S(k,\That)\\
\Sh^1_{S,\emptyset}(k,\That)&=\Sh^1_{S}(k,\That)/\Sh^1_{\emptyset}(k,\That).
\end{align*}
By \cite[Theorem 4.2]{Schlank}
$$
\Ch^1_S(T)\simeq \Hom(\Sh^1_{S,\emptyset}(k,\That),\QQ/\ZZ).
$$
Now $\Sh^1_{S,\emptyset}(k,\That)$ is by definition is a subquotient of $\Sh^1_\omega(k,\That)$.
Thus in order to prove the theorem it suffices to show that $\Sh^1_\omega(k,\That)=0$.

Denote by $\gg$ the image of $\Gal(L/k)$ in $\Aut(\That)$.
Then $\gg$ is a finite metacyclic group.
We may and shall assume that $\Gal(L/k)=\gg$.
For a place $v$ of $k$, let $D_w\subset \gg$
denote the decomposition group of a place $w$ of $L$ extending $v$.
We write $\gg_v$ for $D_w$, it is defined up to conjugacy in $\gg$.
Since  $\Gal(\kbar/L)$ is a profinite group and $\That$ is a free abelian group,
we have $H^1(L,\That)=0$.
It follows that the inflation homomorphism $H^1(\gg,\That)\to H^1(k,\That)$ is an isomorphism.
Similarly, for each $v$ the homomorphism $H^1(\gg_v,\That)\to H^1(k_v,\That)$ is an isomorphism.
Thus we obtain an isomorphism
$$
\ker\left[H^1(\gg,\That)\to\prod_{v\in\sV\smallsetminus S} H^1(\gg_v,\That)\right]\isoto\Sh^1_S(k,\That).
$$
It follows from Chebotarev's density theorem that
$$
\Sh^1_\omega(k,\That)\simeq\ker\left[H^1(\gg,\That)\to\prod_{C} H^1(C,\That)\right],
$$
where $C$ runs over all cyclic subgroups of $\gg$.

Now let $\gg$ be any finite group and $Y$ a finitely generated $\gg$-module.
Let $i\in \ZZ$.
We write $\Hhat^i(\gg,Y)$ for the $i$-th Tate cohomology group.
Following an idea of \cite[page 734]{CTK}, we set
$$
\Sh^i_\Omega(\gg,Y)=\ker\left[\Hhat^i(\gg,Y)\to\prod_C \Hhat^i(C,Y)\right],
$$
where $C$ runs over all cyclic subgroups of $\gg$.
Then $\Sh^1_\omega(k,\That)\simeq\Sh^1_\Omega(\gg,\That)$.
In order to prove Theorem \ref{thm:Htor-metacyclic} it suffices to show that
$\Sh^1_\Omega(\gg,\That)=0$, which follows from the next lemma.
\end{proof}

\begin{lemma}[B. Kunyavski\u\i, private communication]
Let $\gg$ be a finite group and $Y$ a finitely generated $\gg$-module.
If $\gg$ is metacyclic, then $\Sh^i_\Omega(\gg,Y)=0$ for all $i\in \ZZ$.
\end{lemma}

\begin{proof}
Let $y\in\Sh^i_\Omega(\gg,Y)\subset\Hhat^i(\gg,Y)$.
For a subgroup $\hh\subset\gg$ let $\Res_\hh(y)\in\Hhat^i(\hh,Y)$
denote the restriction of $y$ to $\hh$.
Since $y\in\Sh^i_\Omega(\gg,Y)$, we have $\Res_C(y)=0$ for any cyclic subgroup $C\subset\gg$.
Since $\gg$ is metacyclic, every Sylow subgroup of $\gg$ is cyclic.
We see that $\Res_S(y)=0$ for for any Sylow subgroup $S$ of $\gg$.
By \cite[Chapter IV, Section 6, Corollary 4 of Proposition 8]{ANT}
we have $y=0$.
Thus $\Sh^i_\Omega(\gg,Y)=0$.
This completes the proofs of the lemma and of Theorem \ref{thm:Htor-metacyclic}.
\end{proof}


\end{document}